\documentclass[a4paper,10pt,onecolumn,accepted=2026-07-07]{quantumarticle}
\pdfoutput=1
\usepackage{amsmath,amssymb,amsthm,verbatim,hyperref}
\usepackage[mathscr]{eucal}
\usepackage[T1]{fontenc}

\theoremstyle{plain}
\newtheorem{theorem}{Theorem}

\newtheorem{proposition}[theorem]{Proposition}
\newtheorem{corollary}[theorem]{Corollary}
\theoremstyle{definition}
\newtheorem*{remark}{Remark}

\newcommand{\A}{\mathscr{A}}
\newcommand{\M}{\mathscr{M}}
\newcommand{\N}{\mathscr{N}}
\newcommand{\E}{\mathbb{E}}
\newcommand{\Ha}{\mathcal{H}}
\newcommand{\1}{\mathbbm{1}}

\newcommand{\ro}{\rho}

\newcommand{\f}{\varphi}

\newcommand{\tr}{\operatorname{tr}}
\newcommand{\s}{\operatorname{s}}
\renewcommand{\geq}{\geqslant}
\renewcommand{\leq}{\leqslant}

\begin{document}
\title[Variational formulae for entropies]{Variational formulae for entropy-like functionals for states in von Neumann algebras}
\author{Andrzej \L uczak}
\email{andrzej.luczak@wmii.uni.lodz.pl}
\orcid{0000-0002-3425-7133}
\author{Hanna Pods\k{e}dkowska}
\email{hanna.podsedkowska@wmii.uni.lodz.pl}
\orcid{0000-0003-2652-8850}
\author{Rafa{\l} Wieczorek}
\email{rafal.wieczorek@wmii.uni.lodz.pl}
\orcid{0000-0001-7684-4486}
\address{Faculty of Mathematics and Computer Science,
         \L\'od\'z University,
         S. Banacha 22,
         90-238 \L\'od\'z, Poland}

\begin{abstract}
The paper presents variational formulae for entropy-like functionals, including Segal and R\'enyi entropies, for normal states on semifinite von Neumann algebras. The considered functionals are of the form $\tau(f(h))$ where $\tau$ is a normal faithful semifinite trace on this algebra, $h$ is a positive selfadjoint operator from $L^1(\M,\tau)$, and $f$ is an appropriate convex or concave function. The results cover both finite and semifinite algebras, and the obtained formulae generalise known results, in particular, those concerning relative entropy. Moreover, the connection between quantum entropies and the structure of abelian subalgebras is highlighted, providing new interpretations in the context of quantum information theory.
\end{abstract}

\maketitle

\section*{Introduction}
The von Neumann entropy, introduced in the early development of quantum statistical mechanics, is usually presented for normal states on the type I factor $\mathbb{B}(\Ha)$ \textit{via} a density matrix $\ro$ (a positive trace-class operator with $\tr\,\ro=1$) and the formula
\[
 S_{\mathrm{vN}}(\ro)=-\tr(\ro\log\ro).
\]
For readers accustomed to finite systems, this is the familiar ``matrix'' setting: states are trace-class operators and functional calculus reduces mainly to algebraic operations in finite dimension. When one moves to quantum field theory, the algebraic formulation with observables being elements of an operator algebra is a natural approach (cf. the pioneering paper \cite{HK} and the monographs \cite{H, Ho}). A similar situation occurs in black-hole physics where algebras of type II are basic objects (cf. \cite{CLPW, LW, W}). In this approach, the type classification of $\M$ becomes structurally relevant for entropy, rather than a mere technicality. In particular, for semifinite (finite) algebras (i.e. of type II${}_\infty$ or II${}_1$) equipped with a normal semifinite faithful trace $\tau$, the basic obstacle is that the analogue of a “density matrix” need not be a bounded trace-class operator, instead, this analogue is an operator $h_\varphi$ affiliated with $\M$ --- an element of $L^1(\M, \tau)^+$, and potentially unbounded. This is precisely the setting in which Segal \cite{S} proposed a new notion of entropy for semifinite operator algebras by replacing the canonical trace with $\tau$ and defining entropy through $-\tau(h_\varphi\log h_\varphi)$. As indicated above, moving beyond type I is not only a mathematical generalisation: such algebras appear in a natural way in quantum field theory or black-hole physics, and in these situations a more general notion of entropy becomes necessary.

The paper is devoted to variational formulae for various entropy-like functionals for normal states on a semifinite von Neumann algebra $\M$. By such functionals, we understand functionals of the form $\tau(f(h))$ where $\tau$ is a normal faithful semifinite trace on this algebra, $h$ is a positive selfadjoint operator from $L^1(\M,\tau)$, and $f$ is an appropriate convex (concave) function. In particular, such formulae for the Segal and R\'enyi entropies are obtained.

\section{Preliminaries and notation}
Let $\M$ be a semifinite von Neumann algebra acting on a Hilbert space $\Ha$, with a normal faithful semifinite trace $\tau$, identity $\1$, and predual $\M_*$. By $\M^h$ we shall denote the set of hermitian operators in $\M$, and by  $\M^+$ --- the set of positive operators in $\M$. $\M_*^+$ stands for the set of positive functionals in $\M_*$. By a small abuse of terminology, these functionals will be referred to as \emph{normal states}. The elements $\omega$ in $\M_*^+$ of norm one, i.e. such that $\omega(\1)=1$, will be called \emph{normalised states} (i.e. \emph{states} in customary physical terminology). In what follows, the term \emph{state} will always refer to a \emph{normal state}. For finite $\tau$, we shall assume that $\tau(\1)=1$, thus in this case $\tau$ itself is a normalised state.

A densely defined closed operator $x$ is said to be \emph{affiliated with} $\M$, denoted by $x\eta\M$, if for its polar decomposition
\[
 x=u|x|
\]
$u\in\M$, and the spectral projections of $|x|$ belong to $\M$. Observe that if $x$ is affiliated with $\M$ and $x$ is bounded, then $x\in\M$. For a selfadjoint $h$ affiliated with $\M$, by $W^*(h)$ we shall denote the von Neumann algebra generated by the spectral projections of $h$.

An operator $x$ affiliated with $\M$ is said to be \emph{measurable} if for the spectral decomposition
\[
 |x|=\int_0^\infty t\,e(dt),
\]
we have $\tau(e([t_0,\infty)))<+\infty$ for some $t_0$. It is seen that for finite $\tau$, all operators affiliated with $\M$ are measurable. By $\widetilde{\M}$ we shall denote the *-algebra of all measurable operators endowed with the measure topology, i.e. a translation-invariant topology defined by a fundamental system of neighbourhoods of $0$, $\{N(\varepsilon,\delta): \varepsilon,\delta>0\}$, given by
\begin{align*}
 N(\varepsilon,\delta)=\{x\in\widetilde{\M}:&\text{ there exists a projection $p$ in $\M$ such that}\\ &xp\in\M,\quad\|xp\|_\infty\leq\varepsilon\quad\text{and} \quad\tau(p^\bot)\leq\delta\}.
\end{align*}
Measurable operators $h$ such that
\[
 \tau(|h|)<+\infty
\]
form the Banach space $L^1(\M,\tau)$ with the norm
\[
 \|h\|_1=\tau(|h|), \quad h\in L^1(\M,\tau).
\]

For each $\ro\in\M_*$, there is an operator $D_\ro\in L^1(\M,\tau)$ such that
\[
 \ro(x)=\tau(xD_\ro)=\tau(D_\ro  x), \quad x\in\M.
\]
The correspondence between $\M_*$ and $L^1(\M,\tau)$ defined above is one-to-one and isometric. For a state $\omega$, the corresponding element in $L^1(\M,\tau)^+$ will be denoted by $D_\omega$ and called the \emph{density} of $\omega$, thus
\[
 \omega(x)=\tau(xD_\omega)=\tau(D_\omega x)=\tau\big(D_\omega^{\frac{1}{2}}xD_\omega^{\frac{1}{2}}\big), \quad x\in\M.
\]

For a state $\omega$, its \emph{Segal entropy}, $H(\omega)$, is defined as
\[
 H(\omega)=\tau(D_\omega\log D_\omega)=\int_0^\infty t\log t\,\tau(e(dt)),
\]
where
\begin{equation}\label{sp}
 D_\omega=\int_0^\infty t\,e(dt)
\end{equation}
is the spectral decomposition of $D_\omega$. This definition of entropy with a minus sign before the trace was introduced in \cite{S}. From the inequality $t\log t\geq t-1$, it follows that for finite $\tau$ we have a bound
\[
 H(\omega)\geq\tau(D_\omega)-1=\|\omega\|-1,
\]
so for a normalised state $\omega$, $H(\omega)\geq0$ (but it can happen that $H(\omega)=+\infty$).

For normal states $\omega$ and $\f$, their \emph{relative entropy}, $S(\omega,\f)$, is defined as
\[
 S(\omega,\f)=\tau(D_\omega\log D_\omega-D_\omega\log D_\f).
\]
It is known that for normalised states we have $S(\omega,\f)\geq0$.

Let $\omega$ be a normalised state, and let $\alpha\in(0,1)\cup(1,+\infty)$. The family of \emph{R\'enyi's entropies} $R_\alpha(\omega)$ is defined as
\[
 R_\alpha(\omega)=\frac{1}{\alpha-1}\log\tau(D_\omega^\alpha),
\]
under the assumption that $\tau(D_\omega^\alpha)<+\infty$.

\section{Variational formula for Segal entropy}
In \cite{Pe1}, a variational formula for the relative entropy was obtained which in the case of a finite algebra, $\omega$ --- normalised and faithful, and $\f$ --- faithful has the form
\[
 S(\omega,\f)=\sup_{h\in\M^h}\big(\omega(h)-\log\tau\big(e^{\log D_\f+h}\big)\big).
\]
Since in a finite von Neumann algebra we have
\[
 H(\omega)=S(\omega,\tau),
\]
it follows that in this case
\begin{equation}\label{basic}
 H(\omega)=\sup_{h\in\M^h}\big(\omega(h)-\log\tau\big(e^h\big)\big).
\end{equation}
It is interesting that a formula similar to the equality \eqref{basic} holds also for semifinite algebras, even without the assumption on the faithfulness of $\omega$, with the supremum taken over a larger set of operators.
\begin{theorem}\label{varfor1}
Let $\M$ be a semifinite von Neumann algebra with a normal faithful semifinite trace $\tau$, and let $\omega$ be a normalised state having finite Segal's entropy. Then
\[
 H(\omega)=\sup_{h^*=h\eta\M}\{\tau(D_\omega h)-\log\tau\big(e^h\big):D_\omega h, e^h\in L^1(\M,\tau)\}.
\]
\end{theorem}
\begin{proof}
For arbitrary fixed $h=h^*\eta\M$, such that $D_\omega h, e^h\in L^1(\M,\tau)$ define a normalised state $\f$ on $\M$ by its density
\[
 D_\f=\frac{e^h}{\tau\big(e^h\big)}.
\]
Then we have
\begin{align*}
 0&\leq S(\omega,\f)=\tau(D_\omega\log D_\omega-D_\omega\log D_\f)\\
 &=\tau\big(D_\omega\log D_\omega-D_\omega\big(h-\log\tau\big(e^h\big)\1\big)\big)\\
 &=\tau\big(D_\omega\log D_\omega\big)-\tau\big(D_\omega\big(h-\log\tau\big(e^h\big)\1\big)\big)\\
 &=H(\omega)-\tau(D_\omega h)+\log\tau\big(e^h\big),
\end{align*}
which yields
\[
 H(\omega)\geq\tau(D_\omega h)-\log\tau\big(e^h\big),
\]
showing that
\[
 H(\omega)\geq\sup_{h^*=h\eta\M}\{\tau(D_\omega h)-\log\tau\big(e^h\big):D_\omega h, e^h\in L^1(\M,\tau)\}.
\]
Conversely, let $D_\omega$ have spectral decomposition as in \eqref{sp}. Denote $q=e(\{0\})$, and let $q_1,q_2,\dots$ be projections in $\M$ such that $\sum\limits_{i=1}^\infty q_i=q$, $\tau(q_i)<+\infty$. Let
\[
 f(t)=\begin{cases}
 \log t, & \text{for $t>0$}\\
 0, & \text{for $t=0$}
 \end{cases}.
\]
For arbitrary $\varepsilon>0$, set $\alpha_i=\log\frac{\varepsilon}{2^i\tau(q_i)}$, and define operators $h_1$ and $h_2$ by the spectral decompositions
\[
  h_1=\int_0^\infty f(t)\,e(dt), \qquad h_2=\sum_{i=1}^\infty\alpha_iq_i.
\]
We have
\[
 e^{h_1}=D_\omega+q, \qquad e^{h_2}=\sum_{i=1}^\infty e^{\alpha_i}q_i+q^\bot,
\]
and putting $h=h_1+h_2$, we obtain
\[
 D_\omega h=D_\omega h_1=D_\omega\log D_\omega,
\]
and
\[
 e^h=D_\omega+\sum_{i=1}^\infty e^{\alpha_i}q_i,
\]
since
\[
 D_\omega=q^\bot D_\omega=D_\omega q^\bot,
\]
and
\[
 \sum\limits_{i=1}^\infty e^{\alpha_i}q_i=q\sum\limits_{i=1}^\infty e^{\alpha_i}q_i=\big(\sum\limits_{i=1}^\infty e^{\alpha_i}q_i\big)q.
\]
Consequently,
\[
 \log\tau\big(e^h\big)=\log\big(\tau(D_\omega)+\tau(\sum_{i=1}^\infty e^{\alpha_i}q_i)\big)=\log(1+\varepsilon),
\]
which yields
\begin{align*}
 \tau(D_\omega h)-\log\tau\big(e^h\big)&=\tau(D_\omega\log D_\omega)-\log(1+\varepsilon)\\
 &=H(\omega)-\log(1+\varepsilon)\geq H(\omega)-\varepsilon,
\end{align*}
showing that
\[
 H(\omega)\leq\sup_{h^*=h\eta\M}\{\tau(D_\omega h)-\log\tau\big(e^h\big):D_\omega h, e^h\in L^1(\M,\tau)\}. \qedhere
\]
\end{proof}
Observe that the same reasoning as above leads to the formula \eqref{basic} in the case of a finite algebra and a normalised, not necessary faithful, state.
\begin{theorem}
Let $\M$ be a finite von Neumann algebra with a normal faithful finite trace $\tau$, and let $\omega$ be a normalised state having finite Segal's entropy. Then
\[
 H(\omega)=\sup_{h\in\M^h}\big(\tau(D_\omega h)-\log\tau\big(e^h\big)\big)=\sup_{h\in\M^h}\big(\omega(h)-\log\tau\big(e^h\big)\big).
\]
\end{theorem}

\section{Variational formulae for convex (concave) functions}
Let $\Phi\colon\M\to\M$ be a normal positive linear unital map such that $\tau\circ\Phi=\tau$. Then $\Phi$ may be extended to a bounded positive map from $L^1(\M,\tau)$ into itself, denoted by the same symbol, in particular, it sends the selfadjoint operators to selfadjoint ones. Let $h$ be a selfadjoint operator in $\M$, and let $f$ be an operator convex function defined on an interval containing the spectrum of $h$ and the spectrum of $\Phi(h)$. Then the following Jensen inequality holds (even without the assumption $\tau\circ\Phi=\tau$).
\begin{equation}\label{Jen}
 f(\Phi(h))\leq\Phi(f(h)).
\end{equation}
If only the convexity of $f$ is assumed, then we have the Jensen inequality in its weaker form
\begin{equation}\label{Jen1}
 \tau(f(\Phi(h)))\leq\tau(\Phi(f(h))),
\end{equation}
under the assumption that both sides of the inequality above are defined. (Of course, for concave functions we have reversed inequalities.) A problem with the Jensen inequality arises when we want to extend it to unbounded operators (say from $L^1(\M,\tau)$). For the function $f(t)=t\log t$ it was done in \cite{LP}, while for the function $f(t)=t^\alpha$, $\alpha\in(0,1)\cup(1,+\infty)$ it was done in \cite{LPW1}. In the semifinite case, we restrict attention to the class of (operator) convex or concave functions for which Jensen's inequality in its weaker form holds (the functions $f(t)=t\log t$ and $f(t)=t^\alpha$ being our main examples). We shall denote this class by $\mathfrak{J}$. Observe that for $f$ satisfying the stronger form of Jensen's inequality such that $f(h)\in L^1(\M,\tau)$, and $\Phi$ such that $\tau\circ\Phi=\tau$, both sides of the inequality \eqref{Jen1} are defined. To see this, denote $x=\Phi(f(h))$ and $y=f(\Phi(h))$. For the Jordan decompositions of $x$ and $y$ we have
\[
 0\leq x^+-x^--(y^+-y^-)\leq x^+-y^++y^-,
\]
and denoting by $\s(y^+)$ the support of $y^+$, we get
\[
 0\leq \s(y^+)x^+\s(y^+)-y^+,
\]
which yields
\[
 0\leq\tau(y^+)\leq\tau(\s(y^+)x^+\s(y^+))\leq\tau(x^+)<+\infty,
\]
since
\[
 \tau(x)=\tau(\Phi(f(h)))=\tau(f(h)),
\]
i.e. $x\in L^1(\M,\tau)$. Consequently, $\tau(y)=\tau(y^+)-\tau(y^-)$ is defined.

For a von Neumann subalgebra $\N$ of $\M$ such that $\tau|\N$ is semifinite, denote by $\E_\N$ the normal conditional expectation from $\M$ onto $\N$ such that $\tau\circ\E_\N=\tau$. Then $\E_\N$ can be extended to a bounded (of norm one) linear positive map from $L^1(\M,\tau)$ onto $L^1(\N,\tau|\N)$.
\begin{proposition}\label{conv}
Let $\M$ be a semifinite von Neumann algebra with a normal faithful semifinite trace $\tau$, and let $h$ be a positive selfadjoint operator in $L^1(\M,\tau)$. Then for a continuous convex function $f\in\mathfrak{J}$ defined on $\mathbb{R}_+=[0,+\infty)$ such that $f(h)\in L^1(\M,\tau)$ we have
\begin{align*}
 \tau(f(h))=\sup\{&\tau(f(\E_\A h)):\A - \textit{abelian von Neumann}\\
  &\textit{subalgebra of $\M$ such that $\tau|\A$ is semifinite}\}.
\end{align*}
\end{proposition}
\begin{proof}
Let
\[
 h=\int_0^\infty t\,e(dt)
\]
be the spectral decomposition of $h$. We have for arbitrary $a>0$
\begin{equation}\label{sf}
 +\infty>\tau(h)=\int_0^\infty t\,\tau(e(dt))\geq\int_a^\infty t\,\tau(e(dt))\geq a\tau(e[a,+\infty)),
\end{equation}
showing that for the projection $e((0,+\infty))$ we can find a projection $e([a,+\infty))\leq e((0,+\infty))$ having finite trace. If $\tau(e(\{0\}))<+\infty$, then the algebra $W^*(h)$ is semifinite (though $\tau(e(0,+\infty))=+\infty$ if $\tau$ is not finite). If $\tau(e(\{0\}))=+\infty$, choose projections $q_1,q_2,\dots$ in $\M$ such that $\sum\limits_{i=1}^\infty q_i=e(\{0\})$, $\tau(q_i)<+\infty$, and choose $\lambda_1,\lambda_2,\dots$ in $(0,+\infty)$ such that $e(\{\lambda_i\})=0$. Define a new spectral measure $\tilde{e}$ by the formula
\[
 \tilde{e}|(0,+\infty)\smallsetminus\{\lambda_1,\lambda_2,\dots\}=e|(0,+\infty)\smallsetminus\{\lambda_1,\lambda_2,\dots\}, \qquad \tilde{e}(\{\lambda_i\})=q_i.
\]
Let $\widetilde{\A}$ be the abelian von Neumann subalgebra of $\M$ generated by the spectral measure $\tilde{e}$. Then $\tau|\widetilde{\A}$ is semifinite. Put
\[
 g(t)=\begin{cases}
  t, & \text{for $t\in(0,+\infty)\smallsetminus\{\lambda_1,\lambda_2,\dots\}$}\\
  0, & \text{for $t\in\{\lambda_1,\lambda_2,\dots\}$}
 \end{cases}.
\]
We then have
\[
 h=\int_0^\infty g(t)\,\tilde{e}(dt),
\]
thus $h\in L^1(\widetilde{\A},\tau|\widetilde{\A})$. Hence for the conditional expectation $\E_{\widetilde{\A}}$ (extended to $L^1(\M,\tau)$) we have
\[
 \E_{\widetilde{\A}}h=h.
\]
In either case, there is an abelian von Neumann subalgebra $\widetilde{\A}$ of $\M$ such that $\tau|\widetilde{\A}$ is semifinite, and
\[
 \tau(f(h))=\tau(f(\E_{\widetilde{\A}}h)),
\]
which yields the inequality
\begin{align*}
 \tau(f(h))\leq\sup\{&\tau(f(\E_\A h)):\A - \textit{abelian von Neumann}\\
  &\textit{subalgebra of $\M$ such that $\tau|\A$ is semifinite}\}.
\end{align*}
On the other hand, from Jensen's inequality we get for every $\A$ as above
\[
\tau(f(\E_\A h))\leq\tau(\E_\A f(h))=\tau(f(h)),
\]
showing the claim.
\end{proof}
A counterpart of the proposition above for concave functions is
\begin{proposition}\label{conc}
Let $\M$ be a semifinite von Neumann algebra with a normal faithful semifinite trace $\tau$, and let $h$ be a positive selfadjoint operator in $L^1(\M,\tau)$. Then for a continuous concave function $f\in\mathfrak{J}$ defined on $\mathbb{R}_+=[0,+\infty)$ such that $f(h)\in L^1(\M,\tau)$ we have
\begin{align*}
 \tau(f(h))=\inf\{&\tau(f(\E_\A h)):\A - \textit{abelian von Neumann}\\
  &\textit{subalgebra of $\M$ such that $\tau|\A$ is semifinite}\}.
\end{align*}
\end{proposition}
The proof follows by considering the convex function $-f$ and applying Proposition \ref{conv}.

By $\mathfrak{A}$ we shall denote the family of abelian von Neumann subalgebras of $\M$ generated by countable resolutions of identity $\{p_1,p_2,\dots\}$ for some projections $p_1,p_2,\dots$ in $\M$ such that $\sum\limits_{i=1}^\infty p_i=\1$, and $\tau(p_i)<+\infty$, i.e. for $\A\in\mathfrak{A}$ we have
\begin{equation}\label{A}
 \A=\big\{\sum_{i=1}^\infty\gamma_ip_i:\gamma_i\in\mathbb{C},\,\sup_i|\gamma_i|<+\infty\big\},
\end{equation}
where $p_1,p_2,\dots$ are as above.
\begin{theorem}\label{varfor2}
Let $\M$ be a semifinite von Neumann algebra with a normal faithful semifinite trace $\tau$, let $h$ be a positive selfadjoint operator in $L^1(\M,\tau)$, and let $f\in\mathfrak{J}$ be a continuous convex function defined on $\mathbb{R}_+$ such that $f(h)\in L^1(\M,\tau)$. Then
\begin{align*}
 \tau(f(h))&=\sup\{\tau(f(\E_\A h)):\A\in\mathfrak{A}\}\\
  &=\sup\big\{\sum_{i=1}^\infty f(\alpha_i)\tau(p_i):p_1,p_2,\dots\textit{--- projections in }\M,\\
  &\phantom{sup iaaaa}\sum_{i=1}^\infty p_i=\1,\,\tau(p_i)<+\infty\},
\end{align*}
where
\[
 \alpha_i=\frac{\tau(p_ih)}{\tau(p_i)}.
\]
\end{theorem}
\begin{proof}
For $\A$ of the form \eqref{A}, let
\[
 \E_\A h=\sum_{i=1}^\infty\alpha_ip_i
\]
be the spectral decomposition of $\E_\A h$. From the properties of conditional expectation we get
\[
 \E_\A(p_ih)=p_i\E_\A h=\alpha_ip_i,
\]
thus
\[
 \alpha_i=\frac{\tau(p_ih)}{\tau(p_i)}.
\]
Consequently,
\[
 \tau\big(f(\E_\A h)\big)=\tau\big(\sum_{i=1}^\infty f(\alpha_i)p_i\big)=\sum_{i=1}^\infty f(\alpha_i)\tau(p_i),
\]
which shows that
\begin{align*}
 &\sup\{\tau(f(\E_\A h)):\A\in\mathfrak{A}\}\\
 =&\sup\big\{\sum_{i=1}^\infty f(\alpha_i)\tau(p_i):p_1,p_2,\dots\textit{--- projections in }\M,\\
 &\phantom{supaa}\sum_{i=1}^\infty p_i=\1,\,\tau(p_i)<+\infty\},
\end{align*}
since each $\A\in\mathfrak{A}$ is generated by projections $p_1,p_2,\dots$ such that $\sum\limits_{i=1}^\infty p_i=\1$, $\tau(p_i)<+\infty$.

Let
\[
 h=\int_0^\infty t\,e(dt)
\]
be the spectral decomposition of $h$. Assume first that $\tau(e(\{0\}))=+\infty$, and put $\mu=\tau(e(\cdot))$. The relation \eqref{sf} yields that $\mu$ is a $\sigma$-finite measure on $\mathbb{R}_+^0=(0,+\infty)$.

Let $\{E_1,E_2,\dots\}$ be a partition of $\mathbb{R}_+^0$ into disjoint Borel sets such that $\mu(E_i)<+\infty$. Put
\[
 \alpha_i=\frac{\int_{E_i}t\,\mu(dt)}{\mu(E_i)}=\frac{\tau\big(\int_{E_i}t\,e(dt)\big)}{\tau(e(E_i))}=\frac{\tau(e(E_i)h)}{\tau(e(E_i))}, \quad i=1,2,\dots.
\]
Let $f$ be a (continuous) convex function on $\mathbb{R}_+^0$. From \cite[Theorem~1.2]{G}, it follows that
\begin{align*}
 \int_{\mathbb{R}_+^0} f(t)\,\mu(dt)=\sup\big\{\sum_{i=1}^\infty f(\alpha_i)\mu(E_i):&\{E_1,E_2,\dots\} \textit{ --- partition of $\mathbb{R}_+^0$}\\
  &\textit{such that $\mu(E_i)<+\infty$}\big\}.
\end{align*}
The relation $f(h)\in L^1(\M,\tau)$ implies $f(0)=0$, thus
\[
 \tau(f(h))=\tau\Big(\int_{\mathbb{R}_+}f(t)\,e(dt)\Big)=\int_{\mathbb{R}_+^0}f(t)\,\tau(e(dt))=\int_{\mathbb{R}_+^0}f(t)\,\mu(dt),
\]
which yields
\begin{align*}
 \tau(f(h))=\sup\big\{\sum_{i=1}^\infty f(\alpha_i)\tau(e(E_i)):&\{E_1,E_2,\dots\} \textit{ --- partition of $\mathbb{R}_+^0$}\\
 &\textit{such that $\tau(e(E_i))<+\infty$}\big\}.
\end{align*}
Now choose (arbitrary)  projections $q_1,q_2,\dots$ in $\M$ such that $\sum\limits_{i=1}^\infty q_i=e(\{0\})$, $\tau(q_i)<+\infty$. For the coefficients $\alpha'_i$ we have
\[
 \alpha'_i=\frac{\tau(q_ih)}{\tau(q_i)}=0
\]
since $q_i\leq e(\{0\})$, and thus $f(\alpha'_i)=0$. It follows that
\begin{align*}
 \tau(f(h))=\sup&\big\{\sum_{i=1}^\infty f(\alpha_i)\tau(e(E_i))+\sum_{i=1}^\infty f(\alpha'_i)\tau(q_i):\\
 &\{E_1,E_2,\dots\} \textit{ --- partition of $\mathbb{R}_+^0$ such that $\mu(E_i)<+\infty$}\big\}.
\end{align*}
Since
\[
 \sum_{i=1}^\infty e(E_i)+\sum_{i=1}^\infty q_i=e((0,+\infty))+e(\{0\})=\1,
\]
we obtain the inequality
\begin{align*}
 \tau(f(h))\leq\sup\big\{&\sum_{i=1}^\infty f(\alpha_i)\tau(p_i):p_1,p_2,\dots\textit{--- projections in }\M,\\
 &\sum_{i=1}^\infty p_i=\1,\,\tau(p_i)<+\infty\}.
\end{align*}
On the other hand, the inequality
\[
 \tau(f(h))\geq\sup\{\tau(f(\E_\A h)):\A\in\mathfrak{A}\},
\]
follows from Proposition \ref{conv}.

Now if $\tau(e(\{0\}))<+\infty$, then the measure $\mu=\tau(e(\cdot))$ is $\sigma$-finite on the whole of $\mathbb{R}_+$, and we repeat the above reasoning for partitions $\{E_1,E_2,\dots\}$ of $\mathbb{R}_+$, this time without the requirement $f(0)=0$ (and without introducing the projections $q_1,q_2,\dots$).
\end{proof}
As before, we have a counterpart of the theorem above for concave functions.
\begin{theorem}\label{varfor3}
Let $\M$ be a semifinite von Neumann algebra with a normal faithful semifinite trace $\tau$, let $h$ be a positive selfadjoint operator in $L^1(\M,\tau)$, and let $f\in\mathfrak{J}$ be a continuous concave function defined on $\mathbb{R}_+$ such that $f(h)\in L^1(\M,\tau)$. Then
\begin{align*}
  \tau(f(h))&=\inf\{\tau(f(\E_\A h)):\A\in\mathfrak{A}\}\\
  &=\inf\big\{\sum_{i=1}^\infty f(\alpha_i)\tau(p_i):p_1,p_2,\dots\textit{--- projections in }\M,\\
  &\phantom{inf aaaa}\sum_{i=1}^\infty p_i=\1,\,\tau(p_i)<+\infty\},
\end{align*}
where
\[
 \alpha_i=\frac{\tau(p_ih)}{\tau(p_i)}.
\]
\end{theorem}
As corollaries we obtain the following variational formulae for Segal's and R\'enyi's entropies.
\begin{corollary}\label{cor1}
Let $\M$ be a semifinite von Neumann algebra with a normal faithful semifinite trace $\tau$, and let $\omega$ be a state on $\M$ having finite Segal's entropy. Then
\begin{align*}
 H(\omega)=\sup\big\{&\sum_{i=1}^\infty\tau(p_iD_\omega)(\log\tau(p_iD_\omega)-\log\tau(p_i)):\\
 &p_1,p_2,\dots \textit{ --- projections in }\M,\,\sum_{i=1}^\infty p_i=\1,\,\tau(p_i)<+\infty\big\}
\end{align*}
\begin{align*}
 \phantom{H(\omega)}=\sup\big\{&\sum_{i=1}^\infty\omega(p_i)(\log\omega(p_i)-\log\tau(p_i)):\\
 &p_1,p_2,\dots\textit{ --- projections in }\M,\,\sum_{i=1}^\infty p_i=\1,\,\tau(p_i)<+\infty\big\}.
\end{align*}
\end{corollary}
This result was obtained in \cite[Theorem 4.2]{Pa}. It follows from applying Theorem \ref{varfor2} to the function $f(t)=t\log t$.
\begin{corollary}\label{cor2}
Let $\M$ be a semifinite von Neumann algebra with a normal faithful semifinite trace $\tau$, and let $\omega$ be a normalised state on $\M$ such that $\tau(D_\omega^\alpha)<+\infty$. Then
\begin{enumerate}
 \item[(i)] For $\alpha<1$
\begin{align*}
 R_\alpha(\omega)=&\frac{1}{\alpha-1}\inf\big\{\log\sum_{i=1}^\infty\omega(p_i)^\alpha\tau(p_i)^{1-\alpha}:\\
 &p_1,p_2,\dots\textit{ --- projections in }\M,\,\sum_{i=1}^\infty p_i=\1,\,\tau(p_i)<+\infty\big\}.
\end{align*}
 \item[(ii)] For $\alpha>1$
\begin{align*}
 R_\alpha(\omega)=&\frac{1}{\alpha-1}\sup\big\{\log\sum_{i=1}^\infty\frac{\omega(p_i)^\alpha}{\tau(p_i)^{\alpha-1}}:\\
 &p_1,p_2,\dots\textit{ --- projections in }\M,\,\sum_{i=1}^\infty p_i=\1,\,\tau(p_i)<+\infty\big\}.
\end{align*}
\end{enumerate}
\end{corollary}
The proof follows from applying Theorem \ref{varfor3} and Theorem \ref{varfor2} to the concave function $f(t)=t^\alpha$ for $\alpha<1$, and the convex function $f(t)=t^\alpha$ for $\alpha>1$, respectively.
\begin{remark}
Observe that in the case of a finite von Neumann algebra we may, reasoning as in the proof of Theorem \ref{varfor2}, and using \cite[Remark 1.2]{G}, replace the infinite resolution of identity $\{p_1,p_2,\dots\}$ in Theorems \ref{varfor2} and \ref{varfor3}, and in Corollaries \ref{cor1} and \ref{cor2} by a finite one.
\end{remark}
\begin{theorem}
Let $\M$ be a semifinite von Neumann algebra, and let $\omega$ be a state on $\M$ having finite Segal's entropy. Then the following equalities hold true:
\begin{align*}
 H(\omega)=&\sup\{H(\omega|\A):\A - \textit{abelian von Neumann subalgebra of $\M$}\\
 &\phantom{a sup}\textit{such that $\tau|\A$ is semifinite}\}\\
 =&\sup\{H(\omega\circ\Phi):\Phi - \textit{normal positive linear unital map from an}\\
 &\phantom{a sup}\textit{abelian subalgebra $\A$ of $\M$ such that $\tau|\A$ is semifinite}\\
 &\phantom{a sup}\textit{into $\M$, }\tau\circ\Phi=\tau\}\\
 =&\sup\{H(\omega\circ\Phi): \Phi - \textit{as above, }\A\in\mathfrak{A}\}
\end{align*}
\end{theorem}
\begin{proof}
For an abelian subalgebra $\A$ of $\M$ such that $\tau|\A$ is semifinite, we have
\[
 D_{\omega|\A}=\E_\A D_\omega,
\]
thus the first equality follows from Proposition \ref{conv}. Let $\iota_\A\colon\A\hookrightarrow\M$ be the embedding. Then $\omega|\A=\omega\circ\iota_\A$, which yields
\begin{align*}
 H(\omega)=&\sup\{H(\omega\circ\iota_\A):\textit{$\A-$ abelian subalgebra of $\M$ such that}\\
 &\textit{$\tau|\A$ is semifinite}\}\leq\sup\{H(\omega\circ\Phi):\Phi - \textit{normal positive}\\
 &\textit{linear unital map from an abelian subalgebra $\A$ of $\M$}\\
 &\textit{such that $\tau|\A$ is semifinite into $\M$, $\tau\circ\Phi=\tau$}\}.
\end{align*}
Since on account of \cite[Theorem 11]{LP} we have $H(\omega\circ\Phi)\leq H(\omega)$, the second equality follows. The third equality is a consequence of Corollary \ref{cor1} upon observing that
\[
 \sum_{i=1}^\infty\omega(p_i)(\log\omega(p_i)-\log\tau(p_i))=H(\omega\circ\iota_\A),
\]
where $\A$ is the abelian algebra generated by the projections $p_i$.
\end{proof}

\section*{Author Contributions}

All authors contributed equally to this work.


\begin{thebibliography}{99}
 \bibitem{CLPW}
  V. Chandrasekaran, R. Longo, G. Penington, E. Witten, \emph{An algebra of observables for de Sitter space},
        \href{https://doi.org/10.1007/JHEP02(2023)082}{J. High Energy Phys. \textbf{2023}(2), Article 82 (2023).}
 \bibitem{G}
  S.G. Ghurye, \emph{Information and sufficient sub-fields},
      \href{https://doi.org/10.1214/aoms/1177698033}{Ann. Math. Stat. \textbf{39}(6) (1968), 2056--2066.}
 \bibitem{H}
  R. Haag, \emph{Local Quantum Physics},
    \href{https://doi.org/10.1007/978-3-642-61458-3}{Springer-Verlag Berlin Heidelberg, 1996.}
 \bibitem{HK}
  R. Haag, D. Kastler, \emph{An algebraic approach to quantum field theory},
    \href{https://doi.org/10.1063/1.1704187}{J. Math. Phys. \textbf{5}(7) (1964), 848--861.}
 \bibitem{Ho}
  S.S. Horuzhy, \emph{Introduction to Algebraic Quantum Field Theory},
    \href{https://doi.org/10.1007/978-94-009-1179-6}{Kluwer Academic Publishers, Dordrecht, 1990.}
 \bibitem{LW}
  R. Longo, E. Witten, \emph{A note on continuous entropy},
    \href{https://doi.org/10.4310/PAMQ.2023.v19.n5.a5}{Pure Appl. Math. Quart. \textbf{19}(5) (2023), 2501--2523.}
 \bibitem{LP}
  A. \L uczak, H. Pods\k{e}dkowska, \emph{Mappings preserving Segal's entropy in von Neumann algebras},
    \href{https://doi.org/10.5186/aasfm.2019.4439}{Ann. Acad. Sci. Fenn. Math. \textbf{44} (2019), 769--789.}
 \bibitem{LPW1}
  A. \L uczak, H. Pods\k{e}dkowska, R. Wieczorek, \emph{Mappings preserving quantum R\'enyi's entropies in von Neumann algebras}, to appear in Publ. Res. Inst. Math. Sciences (Kyoto) (2026),
    \href{https://doi.org/10.48550/arXiv.2302.02282}{arXiv:2302.02282.}
 \bibitem{Pe1}
  D. Petz, \emph{A variational expression for the relative entropy},
    \href{https://doi.org/10.1007/BF01225040}{Comm. Math. Phys. \textbf{114} (1988), 345--349.}
 \bibitem{Pa}
  A.R. Padmanabhan, \emph{Probabilistic aspects of von Neumann algebras},
    \href{https://doi.org/10.1016/0022-1236(79)90058-2}{J. Funct. Anal. \textbf{31} (1979), 139--149.}
 \bibitem{S}
  I.E. Segal, \emph{A note on the concept of entropy},
    \href{https://doi.org/10.1512/IUMJ.1960.9.59036}{J. Math. Mech. \textbf{9}(4) (1960), 623-629.}
 \bibitem{W}
  E. Witten, \emph{Gravity and the crossed product},
    \href{https://doi.org/10.1007/JHEP10(2022)008}{J. High Energy Phys. \textbf{2022}(10), Article 8 (2022).}
 \end{thebibliography}
\end{document}